\documentclass[12pt]{article}
\usepackage[pagebackref,colorlinks=true,urlcolor=blue,linkcolor=red,citecolor=red]{hyperref}
\usepackage{epsfig,graphics, graphicx}
\usepackage{subcaption, comment, bm}
\usepackage[percent]{overpic}
\usepackage{float}
\usepackage{amssymb,bm}
\usepackage{amsthm}
\usepackage{color}
\usepackage[]{amsmath}
\usepackage[]{amsfonts}
\usepackage[]{fancyhdr}
\usepackage[]{graphicx,wrapfig}
\usepackage{enumitem}
\usepackage[utf8]{inputenc}
\usepackage[T1]{fontenc}
\usepackage{mathtools}
\usepackage{array}

 \usepackage{pdfsync,verbatim,stmaryrd}
 \newcommand{\Beq}{\begin{equation}}
 \newcommand{\Eeq}{\end{equation}}
 \newcommand{\beq}{\begin{equation*}}
 \newcommand{\eeq}{\end{equation*}}
 \newcommand{\bal}{\begin{align}}
 \newcommand{\eal}{\end{align}}
 
 \newcommand{\D}{\mathrm{d}}

 \newcommand{\Lc}{\mathcal{L}}
 \newcommand{\Mc}{\mathcal{M}}

 \newcommand{\Rc}{\mathcal{R}}
 \newcommand{\Sc}{\mathcal{S}}
 \newcommand{\Tc}{\mathcal{T}}

 \newcommand{\Zc}{\mathcal{Z}}

 \newcommand{\Rb}{\mathbb{R}}
 \newcommand{\Sb}{\mathbb{S}}

 \renewcommand{\o}{\omega}

 \usepackage{amsfonts}
 \newcommand{\FR}{\mathbb{R}} 

 \DeclareMathAlphabet{\bi}{OML}{cmm}{b}{it}
 \DeclareMathAlphabet{\bcal}{OMS}{cmsy}{b}{n}
 \DeclareMathAlphabet{\brmn}{OT1}{cmr}{bx}{n}
 \usepackage{amsopn}
 \usepackage{amsmath,amsthm,amssymb}

\newtheorem{remark}{Remark}

 \newtheorem{theorem}{Theorem}
 
 \newtheorem{lemma}{Lemma}

 \newtheorem{definition}{Definition}

\title{\vspace{-1cm} Radon Transform over Tensor Fields: Injectivity, Range, and Unique Continuation Principle}

\author{Rohit Kumar Mishra\thanks{\textbf{Corresponding author}, Department of Mathematics, Indian Institute of Technology, Gandhinagar, Gujarat, India. \url{rohit.m@iitgn.ac.in}, \url{rohittifr2011@gmail.com}} \and Chandni Thakkar\thanks{Department of Mathematics, Indian Institute of Technology, Gandhinagar, Gujarat, India. \url{thakkar_chandni@iitgn.ac.in}}}

\begin{document}
\maketitle
\date{}
\begin{abstract}
A central objective in inverse problems arising in integral geometry is to understand the kernel characterization, inversion formulas, stability estimates, range characterization, and unique continuation properties of integral transforms. In this paper, we study all these aspects for Radon transforms acting on symmetric $m$-tensor fields in $\mathbb{R}^n$. Our results show that these transforms admit a coherent analytic structure, extending several key features of the classical Radon transform and tensor ray transforms to a broader geometric setting.
\end{abstract}
\vspace{2mm}

\section{Introduction}\label{Sec:Introduction}

The Radon transform, introduced by Johann Radon in 1917 \cite{Radon_transform}, associates to a function its integrals over affine hyperplanes in $\mathbb{R}^n$. It is a fundamental object in integral geometry and inverse problems. Its main analytic features including kernel characterization, inversion, range description, stability, and support theorems are well understood; see \cite{Helgason_1999}.

A natural extension of this framework is to consider transforms acting on vector and tensor fields. In this setting, new phenomena appear at the level of injectivity and kernel structure. Even for the standard ray transform of symmetric tensor fields, uniqueness fails: the transform annihilates potential fields, and only the solenoidal component can be recovered. This leads to the study of decomposition and injectivity questions, which form the foundation of tensor tomography; see \cite{Sharafutdinov_Book,Sharafutdinov_Saint-Venant}.

A large body of work has been devoted to ray transforms of tensor fields, where the integration is carried along straight lines, and their variants. We cite a few representative works below; for a broader overview, we refer the reader to the references therein. Kernel and injectivity properties, including solenoidal injectivity and support theorems, have been studied in \cite{Sharafutdinov_Book,Sharafutdinov_Novikov_2007,Vertgeim_2000}. The transverse ray transform and its properties have been investigated in \cite{TRT_injectivity,Anuj_TRT_Support,Lionheart_Naeem_2016}, along with results for restricted data problems; see \cite{Rohit_2024}. Mixed ray transforms have also been studied in various settings; see \cite{Joonas_2023,UCP_MiRT}.

Stability estimates for tensor tomography are closely related to identities of Reshetnyak type and related techniques; see \cite{Sharafutdinov_Reshetnyak_formula,Katsevich_Schuster,Holman_2013}. Range characterization problems have been studied for tensor and moment-type transforms in \cite{Rohit_2021,Momentum_ray_transform_2_2020}, with further developments in the analysis of associated normal operators \cite{normal_operator_MRT}. Unique continuation properties for ray transforms have been investigated in \cite{UCP_2022}. Microlocal aspects, describing the propagation of singularities and invertibility at the level of wavefront sets, are studied in \cite{Microlocal_2018,Microlocal_2021,Uhlmann_2008,microlocal_MiRT}. Together, these results show that ray transforms admit a rich and well-developed analytic theory.

In parallel, Radon transforms for tensor fields have recently attracted considerable interest. Such transforms arise naturally in many applications. Compared to the classical Radon transform, these generalized transforms are less studied. Recently, generalized Radon transforms acting on vector and symmetric tensor fields in $\mathbb{R}^n$ have been introduced and inversion formulas have been obtained in \cite{Generalised_Radon_inversion,Polyakova_Svetov_Gen_Rad_tensor_fields, Kunyansky_2023, Polyakova_Svetov_normal_Radon_numerical}. These works indicate that such transforms share several features with ray transforms and classical Radon transform. However, beyond inversion, several fundamental analytic questions remain open.

The purpose of this paper is to develop a systematic study of generalized Radon transforms acting on symmetric $m$-tensor fields in $\mathbb{R}^n$. In particular, we study the following:
\begin{itemize}
    \item kernel characterization and inversion,
    \item stability estimates and range description,
    \item unique continuation results.
\end{itemize}
While each of these aspects is well understood for the classical Radon transform and its various variants, they have not been treated in a unified way for generalized Radon transforms on tensor fields. Our results show that these transforms admit a coherent analytic structure, extending several key features of the classical theory to this broader setting. 

The paper is organized as follows. In Section \ref{section: def and notations}, we introduce the transform and fix the notation. Section \ref{section: kernel and inversion} is devoted to the characterization of the kernel and inversion. In Section \ref{section: Stability and Range}, we study stability estimates and range characterization. Finally, in Section \ref{section: UCP}, we address uniqueness and non-uniqueness for the generalized Radon transform.

\section{Definitions and notations} \label{section: def and notations}
In this section, we present the definitions and notations needed to state the main results and for further discussion in the article. We start by introducing various Sobolev spaces by completing Schwartz spaces with specific norms. We introduce the classical Radon transform and discuss some of its basic properties. Then, we define the generalized Radon transform of symmetric $m$-tensor fields in $\Rb^n$ as an extension of longitudinal and transverse Radon transforms (studied in \cite{Generalised_Radon_inversion}).\\

\noindent For a given $\omega \in \mathbb{S}^{n-1}$ and $p \in \mathbb{R}$, the hyperplane perpendicular to $\omega$ and at a signed distance $p$ from the origin is denoted by $H_{\omega, p}$ and is given by
\[ H_{\omega,p} := \{x\in\mathbb{R}^n : \langle x,\omega\rangle = p\}.\]
\noindent The tangent bundle and the unit tangent bundle of the unit sphere $\mathbb{S}^{n-1}$ are denoted by $T\Sb^{n-1}$ and $\widetilde{T} \mathbb{S}^{n-1}$, respectively. Specifically, these spaces are defined as follows:
\[T\mathbb{S}^{n-1}
:= \{(u,\omega)\in \mathbb{R}^{n}\times\mathbb{S}^{n-1}
: \langle u, \omega\rangle = 0\}\ \mbox{ and }\ 
\widetilde{T} \mathbb{S}^{n-1}
:= \{(u,\omega)\in \mathbb{S}^{n-1}\times\mathbb{S}^{n-1}
: \langle \omega,u\rangle = 0\}.
\]
\noindent For notational convenience, and in order to define the integral transforms considered in this work, we introduce the parameter space
\[ \mathcal{Z} := \widetilde{T}\mathbb{S}^{n-1}\times\mathbb{R} = \{(u,\omega,p)\in \mathbb{S}^{n-1}\times\mathbb{S}^{n-1}\times\mathbb{R} : \langle \omega,u\rangle = 0\}. \]

\noindent Throughout this work, we will be dealing with functions and tensor fields defined on the spaces $\mathbb{S}^{n-1}\times\mathbb{R}$ and $\mathcal{Z}$. We will repeatedly use the Fourier transform and convolution defined over these spaces. Therefore, we start with a brief discussion on how the Fourier transform and convolution act on functions/tensor fields over the spaces $\mathbb{S}^{n -1}\times\mathbb{R}$ and $\mathcal{Z}$. We use the notations $\Sc(\mathbb{S}^{n-1}\times\mathbb{R})$ and $\Sc(\mathcal{Z})$ to denote the Schwartz space defined on respective spaces. In both cases, the Fourier transform and convolution operators are taken with respect to the real variable $p$ only.\\

\noindent Let $f \in \mathcal{S}(\mathbb{S}^{n-1}\times\mathbb{R})$. We define the Fourier transform of $f$ by
\[
\widehat{f}(\omega,\sigma)
:= \frac{1}{{(2 \pi)}^{\frac{1}{2}}} \int_{\mathbb{R}} f(\omega,p)\, e^{- i p \sigma}\, dp,
\qquad (\omega,\sigma)\in \mathbb{S}^{n-1}\times\mathbb{R}.
\]
Similarly, given $f,g \in \mathcal{S}(\mathbb{S}^{n-1}\times\mathbb{R})$, their convolution is defined by
\[ (f * g)(\omega,p) := \int_{\mathbb{R}} f(\omega,p-q)\, g(\omega,q)\, dq. \]
\noindent Analogously, for a function
$f \in \mathcal{S}(\mathcal{Z})$, we define the Fourier transform by
\[
\widehat{f}(u,\omega,\sigma)
:= \frac{1}{{(2 \pi)}^{\frac{1}{2}}} \int_{\mathbb{R}} f(u,\omega,p)\, e^{- i p \sigma}\, dp,
\qquad (u,\omega,\sigma)\in \widetilde{T}\mathbb{S}^{n-1}\times\mathbb{R}.
\]
The convolution of two functions $f,g \in \mathcal{S}(\mathcal{Z})$ is defined by
\[
(f * g)(u,\omega,p)
:= \int_{\mathbb{R}} f(u,\omega,p-q)\, g(u,\omega,q)\, dq.
\]

\noindent Following \cite{Sharafutdinov_Reshetnyak_formula, Sharafutdinov_Sobolev_Radon} and using the Fourier transform defined above, we now introduce a family of weighted Sobolev-type Hilbert spaces on $\mathbb{R}^n$, as well as the associated spaces on $\mathbb{S}^{n-1}\times\mathbb{R}$ and $\mathcal Z$, which are essential for the analysis of inversion and stability question we address here. For $s\in\mathbb{R}$ and $t>-n/2$, we define the Hilbert space $H^s_t(\mathbb{R}^n)$ as the completion of the Schwartz space $\mathcal{S}(\mathbb{R}^n)$ with respect to the norm
\[ \|f\|^2_{H^s_t(\mathbb{R}^n)} := \int_{\mathbb{R}^n} |y|^{2t}(1+|y|^2)^{s-t} |\widehat f(y)|^2\,dy. \]
When $t=0$, this norm coincides with the usual Sobolev $H^s$-norm 
\[ \|f\|^2_{H^s(\mathbb{R}^n)} = \|f\|^2_{H^s_0(\mathbb{R}^n)}
= \int_{\mathbb{R}^n} (1+|y|^2)^s |\widehat f(y)|^2\,dy\]
and the space $H^s_0(\mathbb{R}^n)$ coincides with the well known classical Sobolev space $H^s(\mathbb{R}^n)$. \\

\noindent In an analogous manner, we introduce weighted Sobolev norms for functions defined on the spaces $\mathbb{S}^{n-1}\times\mathbb{R}$ and $\mathcal Z$ as follows:
\begin{align*}
\|f\|^2_{H^s_t(\mathbb{S}^{n-1}\times\mathbb{R})} &:= \frac{1}{2(2\pi)^{n-1}} \int_{\mathbb{S}^{n-1}}\int_{\mathbb{R}} |s|^{2t}(1+|s|^2)^{s-t} |\widehat f(\omega,s)|^2\,ds\,d\omega,\ \mbox{ for } f \in \Sc(\Sb^{n-1}\times \Rb)\\
\|f\|^2_{H^s_t(\mathcal Z)} &:= \frac{1}{2(2\pi)^{n-1}} \int_{\mathbb{S}^{n-1}} \int_{\mathbb{S}^{n-1}\cap \omega^\perp} \int_{\mathbb{R}} |s|^{2t}(1+|s|^2)^{s-t} |\widehat f(\omega,s,u)|^2 \,ds\,du\,d\omega\ \mbox{ for } f \in \Sc(\Zc). 
\end{align*}
The spaces $H^s_t(\Sb^{n-1}\times \Rb)$ and $H^s_t(\Zc)$ are defined as the completion of $\Sc(\Sb^{n-1}\times \Rb)$ and $\Sc(\Zc)$ with respect to norms $\|\cdot\|_{H^s_t(\mathbb{S}^{n-1}\times\mathbb{R})}$ and $\|\cdot\|_{H^s_t(\Zc)}$, respectively.\\

\noindent Next, we briefly recall the definition of the classical Radon transform acting on scalar fields. It maps a Schwartz function $f \in \mathcal{S}(\mathbb{R}^n)$ to its integrals over affine hyperplanes. More precisely, the \emph{Radon transform}
\[
\mathcal{R} : \mathcal{S}(\mathbb{R}^n) \longrightarrow
\mathcal{S}(\mathbb{S}^{n-1} \times \mathbb{R})
\]
is defined by
\[
\mathcal{R}f(\omega,p)
:= \int_{H_{\omega,p}} f(x)\, ds,
\]
where $ds$ denotes the standard Euclidean surface measure on the hyperplane $H_{\omega,p}$.\\

\noindent The Radon transform admits a unique extension to a bijective isometry between Hilbert spaces; see \cite[Theorem~2.1]{Sharafutdinov_Reshetnyak_formula}. More precisely,
\[
\mathcal{R} : H^s_t(\mathbb{R}^n)
\longrightarrow
H^{\,s + \frac{n-1}{2}}_{\,t + \frac{n-1}{2},\,e}
(\mathbb{S}^{n-1} \times \mathbb{R}).
\]
As a consequence, a function $f \in H^s_t(\mathbb{R}^n)$ is uniquely determined by its Radon transform. Here, $\displaystyle H^{\,s + \frac{n-1}{2}}_{\,t + \frac{n-1}{2},\,e}
(\mathbb{S}^{n-1} \times \mathbb{R})$
denotes the subspace of $H^{\,s + \frac{n-1}{2}}_{\,t + \frac{n-1}{2}} (\mathbb{S}^{n-1} \times \mathbb{R})$ consisting of functions $\varphi$ satisfying the evenness condition $\displaystyle \varphi(-\omega,-p) = \varphi(\omega,p)$.\\

\noindent The following differentiation property of the Radon transform will be used repeatedly in the upcoming discussions:
\begin{equation}\label{eq:property of Radon transform}
\mathcal{R}\!\left(a_i \frac{\partial}{\partial x^i} f(x)\right)(\omega,p) = \langle \omega, a \rangle \frac{\partial}{\partial p} \mathcal{R}f(\omega,p).
\end{equation}
Throughout this article, we employ the Einstein summation convention, whereby repeated indices are implicitly summed over.\\

\noindent Now, we move from scalar fields to the case of tensor fields defined on $\mathbb{R}^n$ (the main object of study) and give some necessary definitions and notations. \\

\noindent Let $\textit{T}^m \Rb^n$ denote the space of $m$-tensor fields defined in $\Rb^n$ and $\mathit{S}^m \Rb^n$ be the space of symmetric $m$-tensor fields. Further, the Schwartz space  $\mathcal{S} (\mathit{S}^m \Rb^n)$ consist of symmetric $m$-tensor fields whose components belong to the Schwartz space $\mathcal{S}(\Rb^n)$ of functions in $\Rb^n$.

\noindent Let $f \in \mathcal{S}(S^m \mathbb{R}^n)$ be a smooth, rapidly decaying symmetric $m$-tensor field. The componentwise Radon transform of $f$, denoted by
$\overline{\mathcal{R}}f$, is the symmetric $m$-tensor field defined by
\[
(\overline{\mathcal{R}}f)_{i_1 \dots i_m}
=
\mathcal{R}\!\left(f_{i_1 \dots i_m}\right).
\]

\noindent We are now ready to define the main integral transform that we study in this manuscript. This transform is the natural generalization of the classical Radon transform for tensor fields in $\Rb^n$. The special cases of this generalised transform have recently been studied by several research groups; please see \cite{Generalised_Radon_inversion, Kunyansky_2023, Polyakova_Svetov_Gen_Rad_tensor_fields}.
\begin{definition}
For given non-negative integers $\ell_1$ and $\ell_2$ such that $\ell_1 +\ell_2 = m$, the \textbf{generalized Radon transform (GRT)} of a symmetric $m$-tensor field is defined as a continuous linear operator  $\Rc^m_{\ell_1,\ell_2}: \mathcal{S} (\mathit{S}^m\Rb^n) \rightarrow \mathcal{S} (\mathcal{Z})$ and is given by:
\begin{equation}\label{eq:def of GRT}
\Rc_{\ell_1\ell_2}^m f(\omega, p, u) = \int_{H_{\omega,p}} \left<f(x), \o^{\odot \ell_1} \odot u^{\odot \ell_2}\right> \,dx = \int_{\omega^\perp} \left< f (p\omega + y), \o^{\odot \ell_1} \odot u^{\odot \ell_2} \right> \,dy. 
\end{equation}
\end{definition}

\begin{enumerate}
    \item[(i)] If $\ell_2 = 0$, then the above transform reduces to the \textit{\textbf{transverse Radon transform}}, denoted by $\Tc^m$.
    \item[(ii)] If $\ell_1 = 0$, then the above the transform reduces to the \textit{\textbf{longitudinal Radon transform}}, denoted by $\Lc^m$.
    \item[(iii)] If both $\ell_1 \neq 0$ and $\ell_2 \neq 0$, the resulting operator is the \textit{\textbf{mixed Radon transform}}, denoted by $\Mc^m$.
\end{enumerate}
\noindent We next introduce two differential operators acting on the space of tensor fields.
\begin{itemize}
    \item The \emph{inner differentiation operator}, also known as the symmetrized covariant derivative,
    \[
        \D : \mathcal{S}(S^{m}\mathbb{R}^n) \longrightarrow \mathcal{S}(S^{m+1}\mathbb{R}^n),
    \]
    is defined by
    \begin{equation*}
        (\D u)_{i_1 \dots i_{m+1}}
        = \sigma(i_1, \dots, i_{m+1})
        \frac{\partial u_{i_1 \dots i_m}}{\partial x_{i_{m+1}}},
    \end{equation*}
    where $\sigma : \textit{T}^m \Rb^n \longrightarrow \mathit{S}^m \Rb^n$ denotes the symmetrization operator
    \begin{align*}
    \sigma (1, \dots, m) u (x_1, \dots, x_m) = \frac{1}{m!} \sum_{\pi \in {\Pi_m}} u (x_{\pi(1)}, \dots, x_{\pi(m)}), \ \Pi_m\mbox{ is the permutation group}.
\end{align*}

    .
    \item The \emph{divergence operator}
    \[
        \delta : \mathcal{S}(S^{m}\mathbb{R}^n) \longrightarrow \mathcal{S}(S^{m-1}\mathbb{R}^n)
    \]
    is given by
    \begin{equation*}
        (\delta u)_{i_1 \dots i_{m-1}}
        = \frac{\partial}{\partial x_i}\, u_{i_1 \dots i_{m-1} i}.
    \end{equation*}
\end{itemize}

\noindent While we have defined our generalized Radon transform in Schwartz space settings, it can be extended to the $H^s_t$ framework using routine density arguments. Finally, we conclude this section with an important decomposition result for tensor fields from \cite{Rohit_Abhishek_support}, which has been used recently in the $H^s_t$-space setting in \cite{Generalised_Radon_inversion} for the inversion of longitudinal and transverse Radon transforms and their moments.

\begin{theorem}\label{thm: Decomposition Result}[Decomposition Result, \cite[Theorem 2]{Generalised_Radon_inversion}]
    For any tensor field $f \in H^s_t(S^m \mathbb{R}^n) (s \in \mathbb{R}, t > -n/2, m \geq 0)$, there exist uniquely defined $v_0, \dots , v_m$ with $v_i \in H^{s + i}_{t + i}(S^{m-i} \mathbb{R}^n)$ for $i = 0,1, \dots, m$ such that
    \begin{align}\label{eq:decomposition of f}
        f = \sum_{i=0}^m \D^iv_i, \quad \mbox{ with }  v_{i} \text{ solenoidal for } 0\leq i \leq  m-1.
    \end{align}
    The estimate 
    \begin{equation*}
        ||v_i||_{H^{s + i}_{t + i}(\mathbb{R}^n)} \leq C ||f||_{H^s_t(\mathbb{R}^n)}; \quad i= 0, \dots, m
    \end{equation*} holds for a constant $C$ independent of $f$. The tensor field $v_0$ is known as the solenoidal part of $f$, and the remaining $\displaystyle \left(\D\sum_{i=1}^m \D^{i-1}v_i\right)$ is the potential part of $f$.
\end{theorem}
\noindent Having established the above framework, we proceed to state and prove the main results of this article in the subsequent sections.
\section{Kernel characterization and inversion of GRT}\label{section: kernel and inversion}
In this section, we begin by giving an equivalent definition of our generalized Radon transform. For the special cases of the longitudinal and transverse Radon transforms, a similar equivalent definition was introduced and extensively used in \cite{Generalised_Radon_inversion}. Working with this equivalent definition, we provide a detailed characterization of the kernel of GRT and then prove its invertibility, enabling reconstruction of the unknown symmetric $m$-tensor field. Our results are consistent with, and extend, the corresponding results obtained in \cite{Generalised_Radon_inversion} for the longitudinal and transverse Radon transforms.\\

\noindent For $\omega \in \mathbb{S}^{n-1}$, denote by $\omega^\perp$ the hyperplane orthogonal to $\omega$ and passing through the origin. Let $\omega_1, \omega_2, \ldots, \omega_{n-1}$ be an orthonormal basis of $\omega^\perp$. Then the collection $\{\omega, \omega_1, \omega_2, \ldots, \omega_{n-1}\}$ constitutes an orthonormal basis of $\mathbb{R}^n$. Using this notation, we now define the equivalent version of the generalized Radon transform.

\begin{definition}
For given non-negative integers $\ell_i$; $i = 1 , \dots , m$ satisfying $\ell_1 + \cdots +\ell_n = m$, the generalized Radon transform of a symmetric $m$-tensor field $\Rc_{\ell_1 \dots \ell_n}^m: \mathcal{S} (\mathit{S}^m\Rb^n) \rightarrow \mathcal{S} (\mathcal{S}^{n - 1} \times \Rb)$ can be defined as follows: 
\begin{align*}
\Rc_{\ell_1 \dots \ell_n}^m f (\omega, p) &= \int_{H_{\omega,p}} \left<f(x), \o_1^{\odot \ell_1} \odot \dots \odot \o_{n - 1}^{\odot \ell_{n - 1}} \odot \o^{\odot \ell_n}\right> \,dx\\
&= \int_{\omega^\perp} \left< f (p\omega + y), \o_1^{\odot \ell_1} \odot \dots \odot \o_{n - 1}^{\odot \ell_{n - 1}} \odot \o^{\odot \ell_n} \right> \,dy.
\end{align*}
\end{definition}
\noindent Hence, once the orthonormal frame $\{\omega, \omega_1, \ldots, \omega_{n-1}\}$ in $\mathbb{R}^n$ is fixed, we omit $u$ from the notation in the definition. In particular, according to the above definition, the case $\ell_n = 0$ yields the longitudinal Radon transform $\Lc^m$, $\ell_n = m$ yields the transverse Radon transform $\Tc^m$, while $0 < \ell_n < m$ gives rise to the mixed Radon transforms $\Mc^m$.

\begin{remark}
    The two definitions of the generalized Radon transform are equivalent, and we use them interchangeably throughout the paper. The intended definition will be clear from the notation. In many cases, results established using one formulation readily extend to the other.
\end{remark}

\noindent The following result provides a characterization of the kernel of the generalized Radon transform.

\begin{theorem}\label{thm: Kernel Characterization}[Kernel Characterization]
    Let $s \in \mathbb{R}$, $t > -\frac{n}{2}$ and $\ell_j \geq 0$; $j = 1, \dots, n$ such that $\ell_1 + \dots + \ell_n = m$. A symmetric $m$-tensor field $f \in H^s_t(S^m \mathbb{R}^n)$ satisfies
    \[\Rc_{\ell_1 \dots \ell_n}^m f = 0 \quad \text{if and only if} \quad f = \sum_{i = 0 \atop i \neq \ell_n}^{m} \D^i v_i \quad \text{ for } \quad v_i \in H^{s+i}_{t+i}(S^{m-i} \Rb^n)\]
    and $v_i$ is divergence-free for $i = 0, 1, \dots, m - 1$.
\end{theorem}
\noindent As a preliminary step to the proof of Theorem~\ref{thm: Kernel Characterization}, we prove the following lemma.

\begin{lemma} \label{lemma: potential tensor fields}
    Let $f = \D^k v$; $k \geq 0$ for some $v \in H^{s + k}_{t + k}(S^{m-k} \mathbb{R}^n)$ be a symmetric $m$-tensor field in $\Rb^n$. Then for $\ell_j \geq 0$; $j = 1, \dots, n$ such that $\ell_1 + \dots + \ell_n = m$, we have $$\Rc_{\ell_1 \dots \ell_n}^m f = \begin{cases}
    0, & \ell_n < k,\\
    \frac{\partial^{k}}{\partial p^{k}} \Rc_{\ell_1 \dots \ell_{n - 1} \ell_{n}^k}^{m - k} v, & \ell_n \geq k
    \end{cases}$$
    where $\ell_{n}^k = \ell_{n} - k$.
\end{lemma}
\begin{proof}
    For the case $k = 0$, the result holds trivially. Let $k \geq 1$, then we have
    \begin{align*}
        \Rc_{\ell_1 \dots \ell_n}^m f &= \Rc_{\ell_1 \dots \ell_n}^m \left(\D^k v\right)\\
        &= \int_{H_{\omega,p}} {\left(\D^k v\right)}_{i_1 \dots i_m} {\left(\o_1^{\odot \ell_1} \odot \dots \odot \o_{n - 1}^{\odot \ell_{n - 1}} \odot \o^{\odot \ell_n}\right)}^{i_1 \dots i_m} \,dx\\
        &= \sigma(i_1, \dots, i_{m}) 
        \int_{H_{\omega,p}} \frac{\partial}{\partial x_{i_{m}}} {(\D^{k - 1} v)}_{i_1 \dots i_{m - 1}} {\left(\o_1^{\odot \ell_1} \odot \dots \odot \o_{n - 1}^{\odot \ell_{n - 1}} \odot \o^{\odot \ell_n}\right)}^{i_1 \dots i_m} \,dx.
    \end{align*}
    If $\ell_n = 0$, then using equation \eqref{eq:property of Radon transform} and the orthogonality of set $\{\omega, \omega_1, \omega_2, \ldots, \omega_{n-1}\}$, we get $\Rc_{\ell_1 \dots \ell_n}^m f = 0$. Otherwise, again using equation \eqref{eq:property of Radon transform}, we have
    \begin{align*}
        \Rc_{\ell_1 \dots \ell_n}^m \left(\D^k v\right) &= \frac{\partial}{\partial p}\int_{H_{\omega,p}} {(\D^{k - 1} v)}_{i_1 \dots i_{m - 1}} {\left(\o_1^{\odot \ell_1} \odot \dots \odot \o_{n - 1}^{\odot \ell_{n - 1}} \odot \o^{\odot \ell_{n} - 1}\right)}^{i_1 \dots i_{m - 1}} \,dx\\ 
        &= \frac{\partial}{\partial p}\Rc_{\ell_1 \dots \ell_{n - 1} \ell_n^{1}}^{m - 1} \left(\D^{k - 1} v\right)
    \end{align*}
    with $\ell_n^{1} = \ell_n - 1$. Repeating the same strategy $k$ times and using the property \eqref{eq:property of Radon transform} multiple times, we get the required result.
\end{proof}

\noindent Now, we proceed to the proof of kernel characterization given in Theorem \ref{thm: Kernel Characterization}.

\begin{proof}[Proof of Theorem \ref{thm: Kernel Characterization}]
We prove the statement in one direction (``if'' part), while the reverse implication (``only if'' part) is a direct consequence of the Decomposition Result \ref{thm: Decomposition Result} and the Inversion Result (Theorem \ref{thm: Inversion}) proved below. For some $0 \leq \ell_n \leq m$, let $$f = \sum_{i = 0 \atop i \neq \ell_n}^{m} \D^i v_i; \quad \delta v_i = 0 \text{ for } i = 0, 1, \dots, m - 1.$$
    Then, we have
    \begin{align} \label{eq: kernel proof}
    \begin{split}
        \Rc_{\ell_1 \dots \ell_n}^m f &= \Rc_{\ell_1 \dots \ell_n}^m \left(\sum_{i = 0 \atop i \neq \ell_n}^{m} \D^i v_i\right)\\
        &= \sum_{i = 0 \atop i \neq \ell_n}^{m} \Rc_{\ell_1 \dots \ell_n}^m \left(\D^i v_i\right).
    \end{split}
    \end{align}
    If $\ell_n = 0$, then using Lemma \ref{lemma: potential tensor fields}, we have
    \begin{align*}
        \Rc_{\ell_1 \dots \ell_n}^m f &= \sum_{i = 1}^{m} \Rc_{\ell_1 \dots \ell_n}^m \left(\D^i v_i\right) = 0.
    \end{align*}
    For $\ell_n \geq 1$, by using equation \eqref{eq: kernel proof} and Lemma \ref{lemma: potential tensor fields}, we get
    \begin{align*}
        \Rc_{\ell_1 \dots \ell_n}^m f &= \sum_{i = 0}^{\ell_n - 1} \Rc_{\ell_1 \dots \ell_n}^m \left(\D^i v_i\right) = \sum_{i = 0}^{\ell_n - 1} \frac{\partial^{i}}{\partial p^{i}} \Rc_{\ell_1 \dots \ell_{n - 1} \ell_{n}^i}^{m - i} v_i. 
    \end{align*}
    Further, we have that $v_i$ is divergence-free for $i = 0, 1, \dots, m - 1$. Hence, using part $(A)$ from the proof of \cite[Theorem 3(b)]{Generalised_Radon_inversion}, we get 
    $$\Rc_{\ell_1 \dots \ell_{n - 1} \ell_{n}^i}^{m - i} v_i = 0$$
    for all $i = 1, \dots, \ell_n - 1$ which finally gives $$\Rc_{\ell_1 \dots \ell_n}^m f = 0.$$
\end{proof}

\noindent Next, we prove the following auxiliary lemma, which will be used to establish our invertibility result later.  

\begin{lemma}\label{lemma: Radon transform of divergence operator}
    For a symmetric tensor field $v$ of order $k$ in $\Rb^n$ and any $1 \leq j \leq k$, we get the following identity:
    \[\Rc_{\ell_1 \dots \ell_n}^{k} (\delta^j v) = \frac{\partial^j}{\partial p^j} \Rc_{\ell_1 \dots \ell_{n - 1} \ell_n + j}^{j + k} v.\]
\end{lemma}

\begin{proof}
    This lemma is proved in slightly different setups in \cite[Lemma~6]{Kunyansky_2023} and \cite[Lemma~2]{Generalised_Radon_inversion}; however, we give a self-contained proof here for completeness.
    \begin{align*}
        \Rc_{\ell_1 \dots \ell_n}^{k} (\delta^j v) &= \int_{H_{\omega,p}} {\left(\delta^j v\right)}_{i_1 \dots i_k} {\left(\o_1^{\odot \ell_1} \odot \dots \odot \o_{n - 1}^{\odot \ell_{n - 1}} \odot \o^{\odot \ell_n}\right)}^{i_1 \dots i_k} \,dx; \quad \ell_1 + \dots + \ell_n = k\\
        &= \int_{H_{\omega,p}} \frac{\partial}{\partial x_{i_{k + 1}}} {\left(\delta^{j - 1} v\right)}_{i_1 \dots i_{k + 1}} {\left(\o_1^{\odot \ell_1} \odot \dots \odot \o_{n - 1}^{\odot \ell_{n - 1}} \odot \o^{\odot \ell_n}\right)}^{i_1 \dots i_k} \,dx\\
        &= \left<\o, e_{i_{k + 1}}\right> \frac{\partial}{\partial p} \int_{H_{\omega,p}} {\left(\delta^{j - 1} v\right)}_{i_1 \dots i_{k + 1}} {\left(\o_1^{\odot \ell_1} \odot \dots \odot \o_{n - 1}^{\odot \ell_{n - 1}} \odot \o^{\odot \ell_n}\right)}^{i_1 \dots i_k} \,dx\\
        &= \frac{\partial}{\partial p} \int_{H_{\omega,p}} {\left(\delta^{j - 1} v\right)}_{i_1 \dots i_{k + 1}} {\left(\o_1^{\odot \ell_1} \odot \dots \odot \o_{n - 1}^{\odot \ell_{n - 1}} \odot \o^{\odot \ell_n + 1}\right)}^{i_1 \dots i_k} \,dx \\
        &= \frac{\partial}{\partial p} \Rc_{\ell_1 \dots \ell_{n - 1} \ell_n + 1}^{j + 1} v.
    \end{align*}
    Repeating the above step $(j - 1)$ times completes the proof of this lemma.
\end{proof}
\noindent Next, we state and prove one of the main results of this work. In simple words, this result tells that a symmetric $m$-tensor field is uniquely determined from its GRT in $\Rb^n$. 
\begin{theorem}\label{thm: Inversion}[Inversion]
    A symmetric $m$-tensor field $f \in H^s_t(\mathit{S}^m \Rb^n); s \in \Rb, t > -\frac{n}{2}$ can be recovered uniquely from the data of generalized Radon transforms $\Rc_{\ell_1 \dots \ell_n}^m f$, with $\ell_1 + \dots + \ell_n = m$. To be more precise, each component $v_{i}$; $i = 0, \dots, m$ of $f$ (from the decomposition result stated in Theorem \ref{thm: Decomposition Result}) can be uniquely recovered from $\Rc_{\ell_1 \dots \ell_{n - 1} i}^m f$ with $\ell_1 + \dots + \ell_{n - 1} = m - i$.
\end{theorem}
\begin{proof}[Proof of Theorem \ref{thm: Inversion}]
From the above Kernel Description Result in Theorem~\ref{thm: Kernel Characterization}, we know that for a fixed index $i$ and integers $\ell_j \ge 0$, $j = 1, \ldots, n-1$, satisfying
\[\ell_1 + \cdots + \ell_{n-1} = m - i,\]
the generalized Radon transform $\mathcal R^{m}_{\ell_1 \dots \ell_{n-1} i} f$ of a symmetric $m$-tensor field $f$ contains information only about the component $v_i$. Indeed, all remaining terms in the decomposition given by Theorem~\ref{thm: Decomposition Result} lie in the kernel of the transform (see Theorem~\ref{thm: Kernel Characterization} below).\\

\noindent The main idea behind this proof is to show that $v_i$ can be uniquely reconstructed from the data of $\mathcal R^{m}_{\ell_1 \dots \ell_{n-1} i} f$. This directly implies that if the data $\mathcal R^{m}_{\ell_1 \dots \ell_{n-1} i} f$ is available for all $0 \le i \le m$, then each component $v_i$, $0 \le i \le m$, can be uniquely recovered, which in turn yields the unique reconstruction of the tensor field $f$. In order to achieve our goal, we show that the componentwise Radon transform $\overline{\Rc}$ of $\delta^i \D^i v_i$ can be recovered from the given data. This further shows that $\delta^i \D^i v_i$ can be uniquely obtained using Radon inversion. Then, $v_i$ is known uniquely as shown in \cite[Theorem 1]{Generalised_Radon_inversion}.\\

\noindent For $i = 0$, it is shown in \cite[Theorem 4]{Generalised_Radon_inversion} that $v_0$ can be uniquely recovered from the longitudinal Radon transform data $\Rc_{\ell_1 \dots \ell_{n - 1} 0}^m f = \Lc_{\ell_1 \dots \ell_{n - 1}}^m f$. Similarly, the recovery of $\overline{\Rc} \left(\delta^m \D^m v_m\right)$, and hence $v_m$, is shown from the transverse Radon transform data $\Rc_{0 \dots 0 m}^m f = \Tc^m f$ in \cite[Theorem 5]{Generalised_Radon_inversion}. Hence, in this proof, we focus on recovering $v_i$ for $i = 1, \dots, m - 1$.\\

\noindent It is well known that the space of symmetric $m$-tensors on $\mathbb{R}^n$ has dimension $\binom{m+n-1}{m}$. Let $\{\omega, \omega_1, \omega_2, \ldots, \omega_{n-1}\}$ be an orthonormal frame in $\mathbb{R}^n$. Then, as shown in \cite[Lemma~6]{Microlocal_2018}, the tensors
\[\omega_1^{\odot \ell_1} \odot \omega_2^{\odot \ell_2} \odot \cdots \odot \omega_{n-1}^{\odot \ell_{n-1}} \odot \omega^{\odot \ell_n}, \qquad \ell_1 + \cdots + \ell_n = m,\]
are linearly independent. Since their cardinality equals the dimension of $S^m \mathbb{R}^n$, they form a basis. At any point $x \in \Rb^n$, we know that $\delta^i \D^i v_i (x)$ is a symmetric $(m - i)$-tensor in $\mathbb{R}^n$ and hence, we can write
\begin{align*}
\delta^i \D^i v_i (x) &= \sum_{\ell_1 + \dots + \ell_n = m - i} \left<\delta^i \D^i v_i (x), \o_1^{\odot \ell_1} \odot \o_2^{\odot \ell_2} \odot \dots \odot \o_{n - 1}^{\odot \ell_{n - 1}} \odot \o^{\odot \ell_n} \right>\\
& \hspace{7cm} \o_1^{\odot \ell_1} \odot \o_2^{\odot \ell_2} \odot \dots \odot \o_{n - 1}^{\odot \ell_{n - 1}} \odot \o^{\odot \ell_n}.
\end{align*}
\noindent So, the componentwise Radon transform $\overline{\Rc}$ of $\delta^i \D^i v_i$ can be written as follows:
\begin{align*}
\overline{\Rc} \left(\delta^i \D^i v_i\right) &= \sum_{\ell_1 + \dots + \ell_n = m - i} \left<\overline{\Rc} \left(\delta^i \D^i v_i\right), \o_1^{\odot \ell_1} \odot \o_2^{\odot \ell_2} \odot \dots \odot \o_{n - 1}^{\odot \ell_{n - 1}} \odot \o^{\odot \ell_n} \right>\\
& \hspace{7cm} \o_1^{\odot \ell_1} \odot \o_2^{\odot \ell_2} \odot \dots \odot \o_{n - 1}^{\odot \ell_{n - 1}} \odot \o^{\odot \ell_n}\\
&= \sum_{\ell_1 + \dots + \ell_{n - 1} = m - i} \left<\overline{\Rc} \left(\delta^i \D^i v_i\right), \o_1^{\odot \ell_1} \odot \o_2^{\odot \ell_2} \odot \dots \odot \o_{n - 1}^{\odot \ell_{n - 1}}\right>\\
& \hspace{7cm} \o_1^{\odot \ell_1} \odot \o_2^{\odot \ell_2} \odot \dots \odot \o_{n - 1}^{\odot \ell_{n - 1}}\\
&+ \sum_{\ell_1 + \dots + \ell_n = m - i \atop \ell_n \neq 0} \left<\overline{\Rc} \left(\delta^i \D^i v_i\right), \o_1^{\odot \ell_1} \odot \o_2^{\odot \ell_2} \odot \dots \odot \o_{n - 1}^{\odot \ell_{n - 1}} \odot \o^{\odot \ell_n} \right>\\
& \hspace{7cm} \o_1^{\odot \ell_1} \odot \o_2^{\odot \ell_2} \odot \dots \odot \o_{n - 1}^{\odot \ell_{n - 1}} \odot \o^{\odot \ell_n}.\\
\end{align*}
The summation above has been split into cases when $\ell_n = 0$ and when $\ell_n \neq 0$. Further, using Lemma \ref{lemma: Radon transform of divergence operator}, we get
\begin{align*}
\overline{\Rc} \left(\delta^i \D^i v_i\right) &= \sum_{\ell_1 + \dots + \ell_{n - 1} = m - i} \frac{\partial^i}{\partial p^i} \left<\overline{\Rc} \left(\D^i v_i\right), \o_1^{\odot \ell_1} \odot \o_2^{\odot \ell_2} \odot \dots \odot \o_{n - 1}^{\odot \ell_{n - 1}} \odot \o^{\odot i} \right>\\
& \hspace{7cm} \o_1^{\odot \ell_1} \odot \o_2^{\odot \ell_2} \odot \dots \odot \o_{n - 1}^{\odot \ell_{n - 1}}\\
&+ \sum_{\ell_1 + \dots + \ell_n = m - i \atop \ell_n \neq 0} \frac{\partial^i}{\partial p^i} \left<\overline{\Rc} \left(\D^i v_i\right), \o_1^{\odot \ell_1} \odot \o_2^{\odot \ell_2} \odot \dots \odot \o_{n - 1}^{\odot \ell_{n - 1}} \odot \o^{\odot \ell_n + i} \right>\\
& \hspace{7cm} \o_1^{\odot \ell_1} \odot \o_2^{\odot \ell_2} \odot \dots \odot \o_{n - 1}^{\odot \ell_{n - 1}} \odot \o^{\odot \ell_n}\\
&= \sum_{\ell_1 + \dots + \ell_{n - 1} = m - i} \left(\frac{\partial^i}{\partial p^i} \Rc_{\ell_1 \dots \ell_{n - 1} i}^{m} \left(\D^i v_i\right) \right)\o_1^{\odot \ell_1} \odot \o_2^{\odot \ell_2} \odot \dots \odot \o_{n - 1}^{\odot \ell_{n - 1}}\\
&+ \sum_{\ell_1 + \dots + \ell_{n} = m - i \atop \ell_n \neq 0} \left(\frac{\partial^i}{\partial p^i} \Rc_{\ell_1 \dots \ell_{n - 1} \ell_n + i}^{m} \left(\D^i v_i\right)\right) \o_1^{\odot \ell_1} \odot \o_2^{\odot \ell_2} \odot \dots \odot \o_{n - 1}^{\odot \ell_{n - 1}} \odot \o^{\odot i}.
\end{align*}
Further, using the known kernel of these transforms from Theorem \ref{thm: Kernel Characterization}, the second summation above becomes zero, and also,
\[\Rc_{\ell_1 \dots \ell_{n - 1} i}^{m} \left(\D^i v_i\right) = \Rc_{\ell_1 \dots \ell_{n - 1} i}^{m} \left(\sum_{i = 0}^m \D^i v_i\right) = \Rc_{\ell_1 \dots \ell_{n - 1} i}^{m} f.\]
Using this relation, we finally get
\begin{align} \label{eq: inversion formula}
\overline{\Rc} \left(\delta^i \D^i v_i\right) &= \sum_{\ell_1 + \dots + \ell_{n - 1} = m - i} \left(\frac{\partial^i}{\partial p^i} \Rc_{\ell_1 \dots \ell_{n - 1} i}^{m} f \right)\o_1^{\odot \ell_1} \odot \o_2^{\odot \ell_2} \odot \dots \odot \o_{n - 1}^{\odot \ell_{n - 1}}.
\end{align}
Since the generalized Radon transform $\Rc_{\ell_1 \dots \ell_{n - 1} i}^{m} f$ is known for all $\ell_j \geq 0$; $j = 1, \dots, n - 1$ with $\ell_1 + \dots \ell_{n - 1} + i = m$, the right hand side of the above equation is completely known. Then, using the inversion of the classical Radon transform and the uniqueness of the solution of the $\delta^i \D^i$ operator, we recover $v_i$ uniquely.
\end{proof}
\section{Stability and Range} \label{section: Stability and Range}
In this section, we derive a version of the Fourier slice theorem for our case. This result is later used to obtain an isometry result (Reshetnyak-type formula) and to describe the range of the generalized Radon transform. \\

\noindent Recall, for a function $f \in \mathcal{S}(\Rb^n)$, the classical Fourier slice theorem states
\begin{equation*}
{(\Rc f)}^\wedge (\o, p) = (2 \pi)^{\frac{n - 1}{2}} \widehat{f}(p\o); \quad p \in \Rb.
\end{equation*}
This result plays a central role in proving many classical results on injectivity, range characterization, and stability. We next derive an analogue of the Fourier slice theorem for the generalized Radon transform.

\begin{lemma} \label{lemma: Fourier slice}
For $f \in \mathcal{S} (\mathit{S}^m\Rb^n)$ and any $0 \leq \ell_1, \ell_2 \leq m$ such that $\ell_1 + \ell_2 = m$, we have
\begin{equation*}
    {(\Rc_{\ell_1 \ell_2}^m f)}^\wedge (\omega, p, u) = (2 \pi)^{\frac{n - 1}{2}} \left<\widehat{f}(p \omega), \omega^{\odot \ell_1} \odot u^{\odot \ell_2}\right>.
\end{equation*}
\end{lemma}

\begin{proof}
The argument is based on an adaptation of the classical Fourier slice theorem to the setting of generalized Radon transforms, and proceeds as follows.
    \begin{align*}
        {(\mathcal{R}_{\ell_1 \ell_2}^m f)}^\wedge (\omega, p, u) &= {(2 \pi)}^{-\frac{1}{2}}\int_\FR e^{-ips} \mathcal{R}_{\ell_1 \ell_2}^m f (\omega, s, u) \,ds\\
        &= {(2 \pi)}^{-\frac{1}{2}}\int_\FR e^{-ips} \left\{\int_{\omega^\perp} f_{i_1 \dots i_{\ell_1} j_1 \dots j_{\ell_2}} (s\omega + y) \omega^{i_1} \dots \omega^{i_{\ell_1}} u^{j_1} \dots u^{j_{\ell_2}} \,dy\right\} \,ds\\
        &= {(2 \pi)}^{-\frac{1}{2}} \omega^{i_1} \dots \omega^{i_{\ell_1}} u^{j_1} \dots u^{j_{\ell_2}} \int_\FR \int_{\omega^\perp} e^{-ips} f_{i_1 \dots i_{\ell_1} j_1 \dots j_{\ell_2}} (s\omega + y) \,dy \,ds.
    \end{align*}
    Taking $x = s \omega + y$, we get
    \begin{align*}
        {(\mathcal{R}_{\ell_1 \ell_2}^m f)}^\wedge (\omega, p, u) &= {(2 \pi)}^{-\frac{1}{2}} \omega^{i_1} \dots \omega^{i_{\ell_1}} u^{j_1} \dots u^{j_{\ell_2}} \int_{\FR^n} e^{-ip(x \cdot \omega)} f_{i_1 \dots i_{\ell_1} j_1 \dots j_{\ell_2}} (x) \,dx\\
        &= {(2 \pi)}^{-\frac{1}{2}} \omega^{i_1} \dots \omega^{i_{\ell_1}} u^{j_1} \dots u^{j_{\ell_2}} \int_{\FR^n} e^{-ix \cdot (p\omega)} f_{i_1 \dots i_{\ell_1} j_1 \dots j_{\ell_2}} (x) \,dx\\
        &= {(2 \pi)}^{\frac{n - 1}{2}} \left< \widehat{f} (p\omega), \omega^{\odot \ell_1} \odot u^{\odot \ell_2} \right>.
    \end{align*}
\end{proof}

\noindent We now introduce the weighted Sobolev spaces $H^s_{t,\ell_1\ell_2}(S^m\mathbb{R}^n)$ for symmetric $m$-tensor fields on $\mathbb{R}^n$. For such a tensor field
$f$, the norm is defined by
\[
\| f \|^2_{H^s_{t,\ell_1\ell_2}(S^m\mathbb{R}^n)}
:= \int_{\mathbb{R}^n} |y|^{2t}(1+|y|^2)^{s-t}
\int_{\mathbb{S}^{n-1}\cap y^\perp}
\left|
\big\langle \widehat f(y),
y^{\odot \ell_1}\odot v^{\odot \ell_2}
\big\rangle
\right|^2
\,dv\,dy.
\]
\noindent The following result shows that the GRT is an isometry over proper Sobolev spaces for symmetric $m$-tensor fields. An equality of a similar kind is established for the classical Radon transform acting on scalar functions in \cite[Theorem 2.1]{Sharafutdinov_Reshetnyak_formula}. 
\begin{theorem}
    For $s \in \Rb$, $t > -n/2$ and any tensor field $f \in \mathcal{S} (\mathit{S}^m\Rb^n)$, the following equality holds:
    \[{\lvert \lvert \mathcal{R}_{\ell_1 \ell_2}^m f \rvert \rvert}_{H^{s + m + \frac{n - 1}{2}}_{t + m + \frac{n - 1}{2}} (\Zc)} = || f ||_{H^s_{t, \ell_1\ell_2} (S^m \FR^n)}.\]
\end{theorem}
\begin{proof} Consider 
    \begin{align*}
        &{\lvert \lvert \mathcal{R}_{\ell_1 \ell_2}^m f \rvert \rvert}_{H^{s + m + \frac{n - 1}{2}}_{t + m + \frac{n - 1}{2}} (\Zc)}\\
        &\qquad= \frac{1}{2(2\pi)^{n-1}} \int_{\Sb^{n-1} \cap \omega^\perp} \int_{\Sb^{n-1}}\int_{-\infty}^\infty |s|^{2 t + 2 m + n - 1} (1 + |s|^2)^{s - t} |\widehat{\mathcal{R}_{\ell_1 \ell_2}^m f}(\o, s, u)|^2 \,ds\, d\o \,du\\
        &\qquad= \frac{1}{(2\pi)^{n-1}} \int_{\Sb^{n-1} \cap \omega^\perp} \int_{\Sb^{n-1}}\int_{0}^\infty |s|^{2 t + 2 m + n - 1} (1 + |s|^2)^{s - t} |\widehat{\mathcal{R}_{\ell_1 \ell_2}^m f}(\o, s, u)|^2 \,ds\, d\o \,du.
    \end{align*}
    Taking $y = s\o$, we get 
    \begin{align*}
        {\lvert \lvert \mathcal{R}_{\ell_1 \ell_2}^m f \rvert \rvert}_{H^{s + m + \frac{n - 1}{2}}_{t + m + \frac{n - 1}{2}} (\Zc)} &= \frac{1}{(2\pi)^{n-1}} \int_{\Sb^{n-1} \cap y^\perp} \int_{\Rb^n} |y|^{2 t + 2 m} (1 + |y|^2)^{s - t} \left\lvert\widehat{\mathcal{R}_{\ell_1 \ell_2}^m f}\left(\frac{y}{|y|}, |y|, \frac{v}{|y|}\right)\right\rvert^2 \,dy \,dv.
    \end{align*}
    Using the Fourier slice theorem for the generalized Radon transform, we get
    \begin{align*}
        {\lvert \lvert \mathcal{R}_{\ell_1 \ell_2}^m f \rvert \rvert}_{H^{s + m + \frac{n - 1}{2}}_{t + m + \frac{n - 1}{2}} (\Zc)} &= \int_{\Sb^{n-1} \cap y^\perp} \int_{\Rb^n} |y|^{2 t} (1 + |y|^2)^{s - t} \left\lvert\left<\widehat{f}(y), y^{\odot \ell_1} \odot v^{\odot \ell_2}\right>\right\rvert^2 \,dy \,dv\\ &= || f ||_{H^s_{t, \ell_1\ell_2} (S^m \FR^n)}.
    \end{align*}
\end{proof}
\noindent Next, we discuss the range characterization for the generalized Radon transform. We begin by recalling the classical range characterization for the Radon
transform acting on scalar functions. This result provides a natural point of comparison for the generalized transforms considered in this work.

\begin{theorem}[Classical range characterization of the Radon transform {\cite{Helgason_1999}}]
Let $g \in \mathcal{S}(\mathbb{S}^{n-1}\times\mathbb{R})$. Then $g$ lies in the range of the classical Radon transform $\mathcal{R}$ acting on
$\mathcal{S}(\mathbb{R}^n)$ if and only if the following conditions hold:
\begin{enumerate}
\item $g(\omega,p)$ is even, i.e.,
\[
g(-\omega,-p) = g(\omega,p),
\]
\item for each $k = 0, 1, 2, \dots$, the moment
\[
\int_{\mathbb{R}} p^k g(\omega,p)\, dp = P_k(\o)
\]
is a homogeneous polynomial of degree $k$ in $\o$.
\end{enumerate}
\end{theorem}

\noindent The next result shows that an analogous description holds for the generalized Radon transforms studied here. In particular, the conditions appearing in the classical case admit natural generalizations that reflect the additional geometric structure of the transform. We now state the corresponding range characterization for the generalized Radon transforms.

\begin{theorem}[Range characterization for the generalized Radon transforms] \label{thm: Range characterization}
    For any $g \in \mathcal{S}(\Zc)$ and $0 \leq \ell_1, \ell_2 \leq m$ with $\ell_1 + \ell_2 = m$, there exists $f \in \mathcal{S}(\mathit{S}^m\Rb^n)$ with $\Rc_{\ell_1\ell_2}^m f = g$ if and if the conditions given below are satisfied:
    \begin{enumerate}
        \item $g(\omega,p)$ is even/odd if $m$ is even/odd, i.e.,
        \[
        g(-\omega,-p,-u) = {(-1)}^m g(\omega,p, u),
        \]
        \item for $k = 0, 1, 2, \dots$, the moment
        \[
        \int_{\mathbb{R}} p^k g(\omega,p,u)\, dp = P_{k + \ell_1, \ell_2}(\o, u)
        \]
        is a polynomial which is homogeneous in $\o$ with degree $k + \ell_1$ and in $u$ with degree $\ell_2$.
    \end{enumerate}
\end{theorem}

\begin{proof}
    We prove that the two conditions stated above are necessary and sufficient for a function $g \in \mathcal{S}(\Zc)$ to lie in the range of the generalized Radon transform. We begin with the proof of necessity; that is, we show that for any fixed $0 \leq \ell_1, \ell_2 \leq m$ with $\ell_1 + \ell_2 = m$ the generalized Radon transforms $\Rc_{\ell_1\ell_2}^m f$ satisfy the above conditions.
    \begin{align*}
        \Rc_{\ell_1\ell_2}^m f (-\o, -p, -u) = \int_{\omega^\perp} \left< f (p\omega + y), {(-\o)}^{\odot \ell_1} \odot {(-u)}^{\odot \ell_2} \right> \,dy = {(-1)}^m \Rc_{\ell_1\ell_2}^m f (\o, p, u).
    \end{align*}
    Further, for any $k = 0, 1, 2, \dots$
    \begin{align*}
        \int_{\mathbb{R}} p^k \Rc_{\ell_1\ell_2}^m f (\o, p, u)\,dp &= \int_{\mathbb{R}} p^k \left(\int_{\omega^\perp} \left< f (p\omega + y), \o^{\odot \ell_1} \odot u^{\odot \ell_2} \right> \,dy\right) \,dp\\
        &= \int_{\mathbb{R}} \int_{\omega^\perp} p^k \left< f (p\omega + y), \o^{\odot \ell_1} \odot u^{\odot \ell_2} \right> \,dy \,dp
    \end{align*}
    Taking $p\omega + y = x$, we get
    \begin{align*}
        \int_{\mathbb{R}} p^k \Rc_{\ell_1\ell_2}^m f (\o, p, u)\,dp &= \int_{\mathbb{R}^n} {\left<x, \o\right>}^k \left< f (x), \o^{\odot \ell_1} \odot u^{\odot \ell_2} \right> \,dx
    \end{align*}
    which is a homogeneous polynomial as required. Now we show that given a function $g \in \mathcal{S}(\Zc)$ that satisfies the two conditions mentioned above, there exists a tensor field $f \in \mathcal{S}(\mathit{S}^m\Rb^n)$ such that $\Rc_{\ell_1\ell_2}^m f = g$. If we prove that there exists $f \in \mathcal{S}(\mathit{S}^m\Rb^n)$ such that 
    \begin{equation} \label{eq: Fourier side range characterization}
        \left<\widehat{f}(p \omega), \omega^{\odot \ell_1} \odot u^{\odot \ell_2}\right> = (2 \pi)^{\frac{1 - n}{2}} \widehat{g}(\o, p, u),
    \end{equation}
    then using Lemma \ref{lemma: Fourier slice} for that tensor field $f$, we get
    \[{(\Rc_{\ell_1 \ell_2}^m f)}^\wedge (\omega, p, u) = (2 \pi)^{\frac{n - 1}{2}} \left<\widehat{f}(p \omega), \omega^{\odot \ell_1} \odot u^{\odot \ell_2}\right>.\]
    This implies
    \[{(\Rc_{\ell_1 \ell_2}^m f)}^\wedge (\omega, p, u) = \widehat{g}(\o, p, u).\]
    Using the injectivity of the Fourier transform, this implies $\Rc_{\ell_1\ell_2}^m f = g$. Now we proceed to prove equation \eqref{eq: Fourier side range characterization}. On the right-hand side, we have
    \begin{align*}
        (2 \pi)^{\frac{1 - n}{2}} \widehat{g}(\o, p, u) = (2 \pi)^{-\frac{n}{2}} \int_{\mathbb{R}} g(\omega,\sigma, u)\, e^{- i p \sigma}\, d\sigma.
    \end{align*}
    Using the expansion 
    \[e^{- i p \sigma} = \sum_{\alpha = 0}^\infty \frac{{(- i p \sigma)}^\alpha}{\alpha !},\]
    we get
    \begin{align*}
        (2 \pi)^{\frac{1 - n}{2}} \widehat{g}(\o, p, u) &= (2 \pi)^{-\frac{n}{2}} \sum_{\alpha = 0}^\infty \frac{{(- i p )}^\alpha}{\alpha !} \int_{\mathbb{R}} \sigma^\alpha g(\omega,\sigma, u)\, d\sigma\\
        &= (2 \pi)^{-\frac{n}{2}} \sum_{\alpha = 0}^\infty \frac{{(- i p )}^\alpha}{\alpha !} P_{\alpha + \ell_1, \ell_2}(\o, u)
    \end{align*}
    using the given condition. Hence, the above expression can be written as $$(2 \pi)^{\frac{1 - n}{2}} \widehat{g}(\o, p, u) = \left<\widehat{f}(p \omega), \omega^{\odot \ell_1} \odot u^{\odot \ell_2}\right>$$
    for some $f \in C^\infty(\mathit{S}^m\Rb^n)$. To complete the proof, we can show that $\widehat{f} \in \mathcal{S}(\mathit{S}^m\Rb^n)$ using the similar arguments as in \cite{Helgason_1999}, which further implies that $f \in \mathcal{S}(\mathit{S}^m\Rb^n)$.
\end{proof}

\section{Unique Continuation Properties}\label{section: UCP}

This section establishes the uniqueness and non-uniqueness results for the transforms introduced above. The analysis is motivated by corresponding results for the classical Radon transform (see \cite{UCP_d-plane}). The main results of this section are stated below:

\begin{theorem} \label{thm: non-uniqueness}
    Suppose $n \geq 2$ is an odd integer and let $U \subset \Rb^n$ be a bounded open set. For any $i \in \{0, 1, \dots, m\}$, there exists a symmetric $m$-tensor field $f \in C_c^\infty(\mathit{S}^m\Rb^n)$ such that
    \[v_i|_U = 0 \quad \text{ and } \quad \Rc^m_{\ell_1 \dots \ell_{n - 1} i} f = 0; \quad \ell_1 + \dots + \ell_{n - 1} + i = m\]
    for all hyperplanes intersecting $U$, but $v_i \not\equiv 0$. Note that $v_i$ here comes from the decomposition result stated in Theorem \ref{thm: Decomposition Result}. This further implies that there exists a non-trivial symmetric $m$-tensor field $f \in C_c^\infty(\mathit{S}^m\Rb^n)$ such that 
    \[f|_U = 0 \quad \text{ and } \quad \Rc^m_{\ell_1 \dots \ell_{n}} f = 0; \quad \ell_1 + \dots + \ell_{n} = m\]
    for all hyperplanes passing through $U$.
\end{theorem}

\begin{theorem} \label{thm: UCP result}
    Assume that $n \geq 2$ is even and $f \in C_c^\infty(\mathit{S}^m\Rb^n)$. Suppose $U \subset \Rb^n$ is bounded open set and for any fixed $i \in \{0, 1, \dots, m\}$, $v_i$ of $f$ (as in the decomposition result from Theorem \ref{thm: Decomposition Result}) vanishes on $U$. If the generalized Radon transform $\Rc^m_{\ell_1 \dots + \ell_{n - 1} i} f$; $\ell_1 + \dots + \ell_{n - 1} + i = m$ also vanishes for all hyperplanes intersecting $U$, then we have
    \[v_i \equiv 0, \quad \text{that is,} \quad f = \sum_{j = 1 \atop j \neq i}^m \D^j v_j.\]
    Consequently, if we have 
    \[f|_U = 0 \quad \text{ and } \Rc^m_{\ell_1 \dots \ell_{n}} f|_U = 0; \quad \ell_1 + \dots + \ell_{n} = m\]
    then this implies $f \equiv 0$.
\end{theorem}

The above two theorems show that uniqueness holds in even dimensions but fails in odd dimensions under the same geometric conditions on the data. We next present a counterexample illustrating the non-uniqueness in odd dimensions and then prove the uniqueness result in the even-dimensional case. The counterexample for the odd-dimensional case is motivated by the counterexample given to prove the non-uniqueness of the classical Radon transform for $n$ odd in Theorem 2.3 in \cite{UCP_d-plane}.

\begin{proof}[Proof of Theorem \ref{thm: non-uniqueness}]
    Let $U$ be a bounded and open subset of $\Rb^n$ and $n \geq 2$ be any fixed odd number. Also, let $i$ be any integer belonging to the set $\{0, 1, \dots, m\}.$ In this proof, we aim to give an example of a non-trival symmetric $m$-tensor field $f \in C_c^\infty(\mathit{S}^m\Rb^n)$ such that $v_i$ of $f$ (as in the decomposition result \ref{thm: Decomposition Result}) vanishes on $U$ and for some $\ell_1, \dots, \ell_{n - 1} \geq 0$ with $\ell_1 + \dots + \ell_{n - 1} + i = m$, we have 
    \[\Rc^m_{\ell_1 \dots \ell_{n - 1} i} f|_U = 0.\]
    Let $m$ be an even integer and $h \in C^\infty(\Rb)$ be a non-trivial function such that
    \begin{itemize}
        \item $h$ is even,
        \item supp$(h) \subset (-a, -1) \cup (1, a)$ for some $a > 1$.
    \end{itemize}
    If $m$ is odd, we just assume $h$ to be odd and the rest of the computation follows. Taking $g(\o, p, u) = h(p)$, it follows from the range characterization of the generalized Radon transform proved in Theorem \ref{thm: Range characterization}, that there exists a symmetric $m$-tensor field $f$ such that 
    \[\Rc_{\ell_1 \dots \ell_{n - 1} i}^m f = g.\]
    Using the known inversion for the generalized Radon transform, we see that 
    \[v_i|_U = 0 \quad \text{ and } \quad \Rc^m_{\ell_1 \dots \ell_{n - 1} i} f|_U = 0\]
    but $f \not\equiv 0$ and hence $v_i \not\equiv 0$.
\end{proof}

\begin{proof}[Proof of Theorem \ref{thm: UCP result}]
    For $f \in C_c^\infty(\mathit{S}^m\Rb^n)$ and any $i \in \{0, 1, \dots, m\}$, if we have 
    \[\Rc^m_{\ell_1 \dots + \ell_{n - 1} i} f|_U = 0\]
    for all $\ell_j \geq 0$; $j = 1, \dots, n - 1$ with $\ell_1 + \dots + \ell_{n - 1} + i = m$, then using the inversion result \eqref{eq: inversion formula}, we get 
    \[\overline{\Rc} \left(\delta^i \D^i v_i\right)|_U = 0.\]
    We also have $v_i|_U = 0$, hence using the UCP result for the classical Radon transform derived in Theorem 2.7 of \cite{UCP_d-plane}, we get
    \[\delta^i \D^i v_i \equiv 0.\]
    Finally, using the uniqueness of the solution of operator $\delta^i \D^i$ proved in \cite[Theorem  3]{Generalised_Radon_inversion}, we have 
    \[v_i \equiv 0.\]
    The above analysis works for all $0 \leq i \leq m$. Now, 
    \[f|_U = 0 \quad \implies \quad v_i|_U = 0; \quad \forall i = 0, 1, \dots, m.\]
    If we have further
    \[\Rc^m_{\ell_1 \dots \ell_{n}} f|_U = 0; \quad \forall \ell_j \geq 0; \quad  j = 1, \dots, n \quad \text{with} \quad \ell_1 + \dots + \ell_n = m.\]
    Then we get
    \[v_i \equiv 0 \quad \forall i = 0, 1, \dots, m, \quad \text{that is} \quad f \equiv 0.\]
    This completes the proof of uniqueness in this case.
\end{proof}
\section{Acknowledgements}\label{sec: acknowledge} CT was supported by the Prime Minister's Research Fellowship from the Government of India.
\section*{Declarations}
\textbf{Conflict of interest:} The authors declare no conflict of interest.
\bibliographystyle{amsplain}
\bibliography{reference}

\end{document}